\documentclass[11pt]{article}

\usepackage[total={5.5in,8.5in}]{geometry}

\usepackage{amsmath}
\usepackage{amssymb}
\usepackage{mathptmx}

\newcommand{\setN}{\mathbb{N}}

\newcommand{\RCAo}{\mathsf{RCA}_0}
\newcommand{\WKLo}{\mathsf{WKL}_0}

\usepackage{amsthm}

\theoremstyle{definition}
\newtheorem{theorem}{Theorem}[section]
\newtheorem{proposition}[theorem]{Proposition}

\newtheorem{corollary}[theorem]{Corollary}
\newtheorem{question}[theorem]{Question}

\usepackage{tikz}

\begin{document}
\title{The logical strength of\\[2pt] K\"{o}nig's edge coloring theorem}
\author{Carl Mummert\footnote{Department of Mathematics, Marshall University. Email:~\texttt{mummertc@marshall.edu}}}
\date{\today}
\maketitle

\begin{abstract}
\noindent K\"{o}nig's edge coloring theorem says that a bipartite graph with
maximal degree $n$ has an edge coloring with no more than $n$ colors.
We explore the computability theory and Reverse Mathematics aspects of this theorem.
Computable bipartite graphs with degree bounded by $n$ have computable
edge colorings with $n+1$ colors, but the theorem that there is an edge
coloring with $n$ colors is equivalent to $\WKLo$ over $\RCAo$. This gives an additional
proof of a theorem of Hirst: $\WKLo$ is equivalent over $\RCAo$
to the principle that every countable bipartite $n$-regular graph is the union of 
$n$ complete matchings.   We describe open questions related to Vizing's edge coloring
theorem and a countable form of Birkhoff's theorem.
\end{abstract}

\section{Introduction}

A 1916 theorem of D.\ K\"{o}nig that says that a bipartite graph with maximal degree $n$ has an edge coloring with $n$ colors.  This note describes the computability and Reverse Mathematics strengths of an infinitary version of this result and related theorems on edge colorings.   There has been significant previous work on this problem, but interesting questions remain open. 

The field of Reverse Mathematics looks to formalize mathematical theorems in second-order arithmetic
to study the logical strength of the formalized principles.  These strengths are often closely
related to the level of noncomputability of solutions to computable instances of the theorems.    We follow Simpson~\cite{Simpson09} for definitions and background on Reverse Mathematics.

The computability question is: if $G$ is a computable bipartite graph with maximal degree~$n$, is there always a computable edge coloring with $n$ colors? If not, what can we say about the noncomputability of the edge colorings that do exist? The answers are interesting. Even if $G$ is highly computable, there may not be a a computable edge coloring with $n$ colors, but there is always a computable edge coloring with a somewhat larger finite number of colors.  In this way, the number of colors we are allowed to use affects how hard it is to find the coloring. 

In a personal communication, Sean Sovine pointed out that a version of K\"{o}nig's edge coloring theorem 
(similar to Corollary~\ref{cor:kon2} below) is used in Fourier analysis, for example in Lemma~7 of Grafakos, He, and Honz\'{\i}k~\cite{GHH18}. In that version, the coloring may use $n^2$ colors. Sovine noticed there is an effective coloring with no more than $2n-1$ colors, as in Proposition~\ref{prop:2nminus1}, and asked if the decomposition can be performed with no more than~$n$ colors. K\"{o}nig's edge coloring theorem shows (in ZFC) that the edge chromatic number is~$n$ if there are no requirements on the effectiveness of the coloring. 

In section~\ref{sec:graph}, we survey the classical results from graph theory and develop a combinatorial equivalence between edge colorings of graphs and decompositions of subsets of $\setN \times \setN$ into partial transversals.  In section~\ref{sec:comp}, we show the edge chromatic number $\chi'(G)$ may vary when we formalize these results in different subsystems of second-order arithmetic. 
$\RCAo$ can prove that if $G$ is a countable bipartite graph with maximum degree $n$ then $\chi'(G) \leq 2n-1$ while $\WKLo$ is equivalent to the proposition that $\chi'(G)$ must be~$n$.  In particular, $\RCAo$ proves that $\chi'(G) \leq 3$  when $n = 2$, while $\WKLo$ is equivalent to the proposition that $\chi'(G) = 2$ in that case. 

There has been significant previous Reverse Mathematics research on graph colorings and matchings. The primary goal of this note is to present a self-contained exposition focusing on K\"onig's edge coloring theorem.   Bean~\cite{Bean76} studied vertex colorings and constructed a computable planar graph that has a vertex coloring with 3 colors, but no computable vertex coloring with a finite number of colors. Kierstead~\cite{Kierstead81} extended Bean's work and also proved that every highly computable, $k$-edge colorable graph has a computable edge coloring with $k+1$ colors. Schmerl~\cite{Schmerl} gives a summary of these results, produces new computable counterexamples, and provides a thorough bibliography of previous results.  At the time, these results were not written in the language of Reverse Mathematics, but the proof methods lead directly to Reverse Mathematics theorems. 

Hirst~\cite{Hirst90} studied the Reverse Mathematics strength of Hall's marriage problem in particular. Hall's marriage theorem is related to the results here in a general sense, and like the theorems here it gives a formal principle equivalent to~$\WKLo$ over~$\RCAo$. The class of graphs to which Hall's theorem applies is different, however.  Hall's theorem applies to graphs with the Hall condition, which may have unbounded degree. We consider graphs with bounded degree which may not satisfy the Hall condition. These
conditions coincide when we look at $n$-regular graphs. 

Shafer~\cite{Shafer11} studied the perfect matching version of Hall's theorem. He also studied a theorem of Birkhoff which says that every doubly stochastic matrix is a convex linear combination of permutation matrices.  K\"onig's edge coloring theorem can be viewed as a natural-number-valued variant in which we consider partial permutation matrices.  Theorem~\ref{thm:match} below was originally proven by Hirst~\cite{Hirst90}; also see Shafer~\cite[Theorem~6.1.8]{Shafer11}.

I would like to thank Sean Sovine, James Schmerl, Michael Schroeder, and Paul Shafer for helpful communications about the results described here. 

\section{Graph theory}\label{sec:graph}

We refer to Diestel~\cite{Diestel00} for general background on graph theory. In this note, a \emph{graph} will always be an undirected graph, which may be countably infinite. We do not generally require graphs to be simple. A graph is \emph{bipartite} if the vertices can be partitioned into two sets $L$ and $R$ so that no edge connects two vertices in $L$ or two vertices in~$R$. A graph is \emph{regular} if every vertex has the same degree, and $n$-regular if that degree is always~$n$.

A (proper) \emph{edge coloring} of a graph $G$, using $k$ colors, is a function that assigns each edge a color from $1$ to $k$ so that no two edges incident to a common vertex have the same color.   The \emph{edge chromatic number} $\chi'(G)$ is the least $k$ such that $G$ has an edge coloring with $k$ colors, if such a $k$ exists.  This number is also the \emph{chromatic index} of the graph.% When $G$ is clear from context this number is denoted~$\chi'$. 

\begin{theorem}[K\"{o}nig 1916]\label{thm:kon}
Suppose that $n \in \setN$ and $G$ is a finite bipartite graph in which each vertex
has degree less than or equal to~$n$. Then the edge chromatic number of $G$ is no more than~$n$.
\end{theorem}

This is a special case of Vizing's theorem, which states that a graph with degree bounded by $n$ must have edge chromatic number less than or equal to~$n+1$.  There is substantial literature on Vizing's theorem, especially on the question of which graphs have edge chromatic number $n$ and which have edge chromatic number $n+1$. These classes of graphs are known as ``class one'' and ``class two'', respectively. See Akbari, Cariolaro, D Chavooshi,  Ghanbari, Zare~\cite{ACCGZ12} for a summary of some results, and  Holyer~\cite{H81} for the difficulty of determining whether a finite graph is class one. 

There are several combinatorially equivalent ways to state K\"{o}nig's edge coloring theorem.  One uses graph matchings. A (partial) \emph{matching} in a graph is a set of edges in which no two edges share a common vertex. A \emph{complete matching} is a matching so that every vertex of the graph is contained in some edge of the matching.  In this language, the theorem is stated as follows. 

\begin{corollary}[K\"{o}nig 1916]\label{cor:kon2}
Suppose that $n \in \setN$ and $G$ is a finite bipartite graph in which each vertex
has degree less than or equal to~$n$. Then $G$ is the union of no more than~$n$ matchings.
If $G$ is a finite bipartite $n$-regular graph then $G$ is the union of $n$ complete matchings. 
\end{corollary}

We are interested in the following infinitary corollary of Theorem~\ref{thm:kon}. 

\begin{corollary}\label{cor:kon1}
Suppose that $n \in \setN$ and $G$ is a countable bipartite graph with vertex set $\setN$ in which each  vertex has degree less than or equal to~$n$. Then $\chi'(G) \leq n$.
\end{corollary}

This corollary can be stated in a combinatorially equivalent way using infinite zero-one matrices, that is, subsets of $\setN \times \setN$. For a subset $S$ of $\setN\times \setN$, a \emph{row} is a set of the form $R_i = S \cap \{(i,j) : j \in \setN\}$ for some $i \in \setN$, and a \emph{column} is a set of the form $C_j = S \cap \{(i,j) : i \in \setN\}$ for some $j \in \setN$.  A \emph{partial transversal}  is a subset of $S$ with no more than one point in each row and no more than one point in each column.
 
\begin{corollary}\label{cor:kon3}
Suppose that $S$ is a subset of $\setN \times \setN$ in which each column and each row of $S$ contains at most $n$ points. Then $S$ can be written as a union of no more than $n$ partial transversals.  
\end{corollary}
\begin{proof}
Given the set $S$, make a bipartite graph $G$ as follows. The vertex set $\setN$ will be viewed abstractly as $\{v_i \colon i \in \setN\} \cup \{w_j \colon j \in \setN\}$.  Put an edge between $v_i$ and $w_j$ if and only if the point $(i,j)$ is in~$S$. This gives a bipartite graph by construction, The degree of a vertex $v_i$ will be exactly
the number of points in row~$i$, and the degree of a vertex $w_j$ will be exactly the number of points in column~$j$. 

By Corollary~\ref{cor:kon1}, we can find a edge coloring of $G$. We then partition $S$ into one partial transversal for each color. Because the coloring is proper, each block of the partition will have at most one point in each row and in each column. 
\end{proof}

A converse version of Corollary~\ref{cor:kon2} also holds. Assume that, whenever $S$ is a subset of $\setN \times \setN$ with no more than $n$ points in each row and each column, there is a partition of $S$ into $n$ partial transversals. Suppose we are given a countably infinite simple bipartite graph $G$ with maximum degree $n$.  Write the two sets of vertices of $G$ as
$L = \{ v_i : i \in \setN\}$ and $R = \{w_j : i \in \setN\}$. Make a set $S \subseteq \setN \times \setN$ by putting $(i,j)$ in $S$ if there
is an edge from $v_i$ to~$w_j$ in~$G$, so the points of $S$ correspond precisely to the edges of~$G$. Then $S$ has no more than $n$ points in each row and column, so we can decompose $S$ into $n$ partial transversals. If we color the edges using these transversals, we have a edge coloring of $G$. 

\section{Computability and Reverse Mathematics}\label{sec:comp}

In this section, we examine the computability and Reverse Mathematics strength of the results of the previous section.  We view countable graphs as having vertex set $V = \setN$, so they are presented as an adjacency function $E\colon V \times V \to \{0,1\}$.  An object is computable from the graph if it is computable from the function~$E$.  A graph is \emph{bounded} if there is a function $f\colon \setN \to \setN$ so
that whenever there is an edge between vertices $v$ and $w$, we have $w \leq f(v)$. This is equivalent to
the existence of a function that tells the exact degree of each vertex, and thus every regular graph is
bounded.  In the literature, a computable graph $G$ is \emph{highly computable} if it is bounded via some computable function~$f$. 

\begin{proposition}\label{prop:2nminus1}
Suppose that $n \in \setN$ and $G$ is a countable graph with vertex set $\setN$ in which each  vertex has degree less than or equal to $n$. Then there is an edge coloring of $G$ with  no more than $2n-1$ colors, and the coloring is computable from~$G$.  
\end{proposition}

\begin{proof}
Fix an enumeration $(e_i : i \in \setN)$ of the edges of $G$.  The computation proceeds in stages. 
At stage $i$ we determine a color for edge~$e_i$. Consider the vertices $v$ and $w$ that are incident to~$e_i$. Each of these is incident to at most $n-1$ other edges, some of which may have been colored at earlier stages. This means at most $2n-2$ colors are forbidden for~$e_i$ because they are already used at those two vertices. Therefore, because $2n-1$ colors are available, we can select a color for~$e_i$ that is not forbidden. 
\end{proof}

The previous proof can be formalized in $\RCAo$ to yield the following corollary. 

\begin{corollary} $\RCAo$ proves that if $G$ is a countable graph in which each vertex
has degree bounded by $n$ then $\chi'(G) \leq 2n-1$.
\end{corollary}

Our next goal is to show that $\WKLo$ is required to reduce the $2n-1$ in the previous corollary to~$n$.
The proof of the following proposition is relatively routine. 

\begin{proposition}[Formalized K\"{o}nig's edge coloring theorem]\label{prop:klforward}
 $\WKLo$ proves that if $G$ is a countable bipartite graph in which each vertex
has degree bounded by $n$ then $\chi'(G)\leq n$.
\end{proposition}

\begin{proof}
The first step is for finite graphs. The proof of K\"{o}nig's edge coloring theorem in the finite case, as presented in Diestel~\cite{Diestel00}, can be directly formalized into $\RCAo$. 

To extend the result to countably infinite graphs, we apply a compactness argument. Let $G$ be a countably infinite graph
with degree bounded by~$n$. If $G$ has a finite number of edges, we can immediately reduce
to the finite case. Therefore, let the set of edges be enumerated in an effective way as $(e(0), e(1), e(2), \ldots)$. Consider the set $T \subseteq [n]^{<\setN}$ of all finite sequences $\sigma$ of $\{1, \ldots, n\}$
that are valid edge colorings of $G$ restricted to the edges in the set $\{e(0), \ldots, e(\operatorname{lh}(\sigma)-1)\}$.  

The set $T$ is a tree and, by the finite case, has at least one sequence of length $k$ for each~$k$.  
Therefore, by bounded K\"{o}nig's lemma, $T$ has an infinite path. This path gives an edge coloring of $G$ with no more than $n$ colors.
\end{proof}

A similar proof method allows us to prove a formalized version of Vizing's theorem in~$\WKLo$.

\begin{proposition}[Formalized Vizing's theorem]
$\WKLo$ proves that if $G$ is a countable graph with maximal degree $n$ then
$\chi'(G) \leq n+1$.
\end{proposition}

\begin{proof}
The method is similar to the previous proof. The finite version of Vizing's theorem
uses an inductive argument that can be formalized in~$\RCAo$. The countable version
can then be obtained in $\WKLo$ using compactness.
\end{proof}

We now show that $\WKLo$ is required for the conclusion of Corollary~\ref{thm:kon}.
%\newpage 

\begin{theorem}\label{thm:1}
The following are equivalent over $\RCAo$:
\begin{enumerate}
\item $\WKLo$.
\item If $G$ is a countable bipartite graph in which each vertex
has degree bounded by $n$ then $\chi'(G) \leq n$.
\item If $G$ is a countable bounded bipartite graph in which each vertex
has degree bounded by $2$ then $\chi'(G) \leq 2$.
\end{enumerate}
\end{theorem}

\begin{proof}
Proposition~\ref{prop:klforward} shows that (1) implies (2), and (2) directly implies~(3). Therefore, 
we show that (3) implies~(1).

We assume (3). To show that $\WKLo$ holds, it is sufficient to show that if we have
two bijections $f, g \colon \setN \to \setN$ with disjoint ranges, there is a separating set $A$
that contains the range of $f$ and is disjoint from the range of~$g$.  

Temporarily fix $k \in \setN$. We first show how to construct a subset $S_k$ of $\setN \times \setN$
so that:
\begin{enumerate}
\item There are at most two points in $S_k$ in each row and column.
\item Any partition of $S_k$ into two partial transversals will allow us to tell whether
to place $k$ into the set $A$.
\end{enumerate}
We begin by placing a point $e^k_1$ at $(0,0)$ and points $e^k_2$ and $e^k_3$ at $(1,1)$ and $(1,2)$, respectively.  If $k$ never enters the 
range of $f$ or $g$, then $S_k$ will be the set $\{e^k_1, e^k_2, e^k_3\}$.  If we see that $f(q) = k$, 
we place a point $e^k_4$ at $(q+2, 0)$ and a point $e^k_5$ at $(q+2, 1)$.   If we see that $g(q) = k$,
we place $e^k_4$ at $(q+2, 0)$ and place $e_5$ at $(q+2, 2)$.  This construction is illustrated in Figure~\ref{fig:fig1}.  The set $S_k$ is uniformly computable relative $k$, $f$ and $g$.  The
graph corresponding to $S_k$ has one vertex for each row and one vertex for each column of $\setN \times \setN$, and as such the
graph is trivially bipartite.

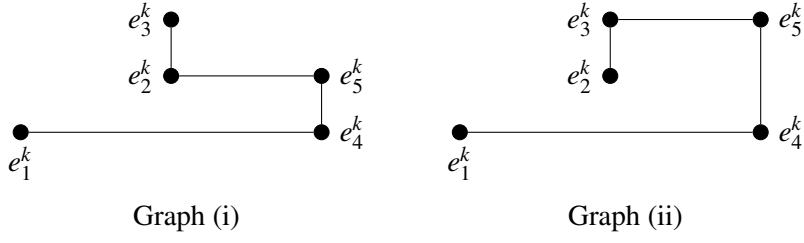
\begin{figure}
\begin{center}
\begin{tabular}{ccc}
  \begin{tikzpicture}[yscale=0.75]
    \tikzstyle{dot}=[draw,circle,inner sep=2pt,fill]
    \node[dot] (A) at (0,0) {};
    \node[yshift=-3mm] at (A.south) {$e^k_1$};
    \node[dot] (B) at (4,0) {};
    \node[xshift=3mm] at (B.east) {$e^k_4$};
    \node[dot] (C) at (4,1) {};
    \node[xshift=3mm] at (C.east) {$e^k_5$};
    \node[dot] (D) at (2,1) {};
    \node[xshift=-3mm] at (D.west) {$e^k_2$};
    \node[dot] (E) at (2,2) {};
    \node[xshift=-3mm] at (E.west) {$e^k_3$};
    
    \draw (A)--(B)--(C)--(D);
	\draw (D)--(E);
  \end{tikzpicture}&
  \qquad &
    \begin{tikzpicture}[yscale=0.75]

    \tikzstyle{dot}=[draw,circle,inner sep=2pt,fill]

    \node[dot] (A) at (0,0) {};
    \node[yshift=-3mm] at (A.south) {$e^k_1$};
    \node[dot] (B) at (4,0) {};
    \node[xshift=3mm] at (B.east) {$e^k_4$};
    \node[dot] (C) at (4,2) {};
    \node[xshift=3mm] at (C.east) {$e^k_5$};
    \node[dot] (D) at (2,1) {};
    \node[xshift=-3mm] at (D.west) {$e^k_2$};
    \node[dot] (E) at (2,2) {};
    \node[xshift=-3mm] at (E.west) {$e^k_3$};
    
    \draw (A)--(B)--(C)--(E);
	\draw (D)--(E);
  \end{tikzpicture}\\
  Graph (i) & & Graph (ii)\\
  \end{tabular}
  \caption{Construction with $S_k = \{e^k_1, e^k_2, \ldots, e^k_5\}$. The edges in the graph shown illustrate which points are in the same row or column.  We will partition $S_k$ into two partial transversals, which is the same as $2$-coloring the vertices of the graph shown. In graph~(i), $e^k_1$ and $e^k_3$ must be in the same partial transversal. In graph~(ii), $e^k_1$ and $e^k_3$ must be in different transversals.}\label{fig:fig1}
  \end{center}
\end{figure}

Suppose that we partition $S_k$ into two partial transversals. By a parity argument, we see
that if $k$ is in the range of $f$ then $e^k_1$ and $e^k_3$ must be in the same block, while
if $k$ is in the range of $g$ then $e^k_1$ and $e^k_3$ must be in different blocks. If $k$ is not 
in the range of $f$ or the range of $g$, then we never add $e^k_4$ or $e^k_5$ to $S$, and
$e^k_3$ can be in either block.

\begin{figure}
\begin{center}
  \begin{tikzpicture}[yscale=0.5]
    \tikzstyle{dot}=[draw,circle,inner sep=2pt,fill]
    \node[dot] (A) at (0,0) {};
    \node[yshift=-3mm] at (A.south) {$e^k_1$};

    \node[dot] (A1) at (2,0) {};
    \node[xshift=4mm] at (A1.east) {$e^k_{1,1}$};

    \node[dot] (A2) at (2,3) {};
    \node[xshift=-4mm,yshift=2mm] at (A2.north west) {$e^k_{1,2}$};

    \node[dot] (A3) at (5,3) {};
    \node[xshift=4mm] at (A3.east) {$e^k_{1,3}$};

    \node[dot] (A4) at (5,7) {};  % should be 6
    \node[xshift=-4mm,yshift=2mm] at (A4.north west) {$e^k_{1,4}$};

% x 2 5 8 = 3q-1
% y 3 6 9  = 3q
% q 1 2 3 

    \draw (A)--(A1)--(A2)--(A3)--(A4);

    \node[dot] (B) at (1,1) {};
    \node[xshift=-3mm] at (B.west) {$e^k_2$};

    \node[dot] (B1) at (3,1) {};
    \node[xshift=4mm] at (B1.east) {$e^k_{2,1}$};

    \node[dot] (B2) at (3,4) {};
    \node[xshift=-4mm,yshift=2mm]at (B2.north west) {$e^k_{2,2}$};

    \node[dot] (B3) at (6,4) {};
    \node[xshift=4mm] at (B3.east) {$e^k_{2,3}$};

    \node[dot] (B4) at (6,7) {};
    \node[yshift=2mm,xshift=4mm]at (B4.north east) {$e^k_{2,4}$};

% q 1 2 3 
% x 3 6 9 = 3q
% y 1 4 7. = 3q+1
       
    \draw (B)--(B1)--(B2)--(B3)--(B4);

    \node[dot] (C) at (1,2) {};
    \node[xshift=-3mm,yshift=2mm] at (C.north west) {$e^k_3$};

    \node[dot] (C1) at (4,2) {};
    \node[xshift=4mm] at (C1.east) {$e^k_{3,1}$};

    \node[dot] (C2) at (4,5) {};
         \node[xshift=-4mm,yshift=2mm] at (C2.north west) {$e^k_{3,2}$};
  
    \node[dot] (C3) at (7,5) {};
    \node[xshift=4mm] at (C3.east) {$e^k_{3,3}$};

    \node[dot] (C4) at (7,8) {};
         \node[xshift=4mm,yshift=2mm] at (C4.north east) {$e^k_{3,4}$};
% q 1 2 3 
% x 1 4. 7.  = 3q+1
% y 2. 5 8 = 3q+2
       
    \draw (C)--(C1)--(C2)--(C3)--(C4);

\draw[line width=1.5pt] (A4)--(B4);

\draw (B)--(C);
\end{tikzpicture}

  \caption{Modified construction with ``staircases''.  If we $2$-color the vertices of the graph,
  $e^k_{1}$ and $e^k_{1,2q}$ must have the same color for each~$q$. Similarly $e^k_{2}$ and $e^k_{2,2q}$
  have the same color, and $e^k_3$ and $e^k_{3,2q}$ have the same color. Because $f(1) = 1$,   $e^k_{1,4}$ is placed on the same row as $e^k_{2,4}$ (shown with a heavy line). If $g(1) = 1$, then $e^k_{1,4}$ would be placed  one row higher, on the same row as $e^k_{3,4}$.  Each row and column contains two points except those corresponding to uncapped ends of staircases.}\label{fig:fig2}
  \end{center}
\end{figure}

In order to obtain the desired result, we need to modify the construction so the corresponding bipartite graph is bounded.  In the construction just shown, we cannot predict the column where 
$e^k_4$ might be placed. To handle this, we will add ``staircase'' paths starting at $e^k_1$, $e_k^2$, and
$e^k_3$ as shown in Figure~\ref{fig:fig2}. Place $e^k_1$, $e^k_2$, and $e^k_3$ as before. Now, at each stage $q \geq 0$, if $k$ has not yet
entered the range of $f$ or $g$, we do the following:
\begin{itemize}
\item Place $e^k_{2, 2q+1}$ at $(3q+3, 3q+1)$, which is three columns to the right of $e^k_{2, 2q-1}$ when $q > 0$, or two columns to the right of~$e^k_2$ when $q =0$.
\item Place $e^k_{2,2q+2}$ at $(3q+3,3q+4)$, which is three rows above $e^k_{2,2q+1}$. %
\item Place $(e^k_{3,2q+1})$ at $(3q+4, 3q+2)$, three columns to the right of $e_3$ or $e^k_{3,2q-1}$.
\item Place $e^k_{3,2q+2}$ at $(3q+4,3q+5 )$, three rows above $e^k_{3,2q+1}$. 
\item Place $e^k_{1,2q+1}$ at $(3q-1,3q)$, three columns to the right of $e^k_{1,2q-1}$.
\item If $f(q) = k$, we place $e^k_{1,2q+2}$ at $(3q-1,3q+4)$, above $e^k_{1,2q+1}$ and to the left of $e^k_{2,2q+2}$.  In this case we have ``capped'' $e^k_{1,2q+2}$ and $e^k_{2,2q+2}$. We stop extending the ``stairs'' at future stages. 
\item 
If $g(q) = k$, we place $e^k_{1,2q+2}$ at $(3q-1,3q+5)$, above $e_{1,2q+1}$ and the left of $e^k_{3,2q+2}$.  In this case we have ``capped'' $e^k_{1,2q+2}$ and $e^k_{3,2q+2}$.
In this case we also stop extending the ``stairs''. 
\item 
Otherwise, we place $e^k_{1,2q+2}$ at $(3q-1, 3q+3)$, three rows above $e^k_{1,2q+1}$.
\end{itemize}

The effect of this construction is that, if $S_k$ is partitioned into two partial transversals,
$e^k_1$ must be in the same block as $e^k_{1,2q+2}$ whenever $e^k_{1,2q+2}$ is added to $S_k$. Similarly
$e^k_2$ and $e^k_{2,2q+2}$ must be in the same block as each other, and $e^k_3$ and $e^k_{3,2q+2}$ must be in the same block as each other. 

Thus, if $e^k_{1,2q+2}$ and $e^k_{2,2q+2}$ are on the same row, they must be in different blocks, and thus $e^k_1$ and $e^k_2$ must be in different blocks.  Hence $e^k_1$ and $e^k_3$ must be 
in the same block.  On the other hand if $e^k_{1,2q+2}$ and
$e^k_{3,2q+2}$ are on the same row then $e^k_1$ and $e^k_3$ must be in different blocks.   Thus, as before, we can learn about whether $q$ is in the range of $f$ or $g$ by looking at whether 
$e^k_1$ and $e^k_3$ are in the same block. 

At the same time, because there are only three possible locations for each $e^k_{1,2q+2}$, and we can effectively find  the correct location from $k$, $q$, $f$, and $g$ the resulting bipartite graph is bounded.

Finally, for the full proof, we need to construct a single set $S \subseteq \setN \times \setN$ that contains a copy of each set $S_k$. To do this, for each prime $p$ let 
\[
D_p = \{ (p^{i+1}, p^{j+1}) : i, j \in \setN\} \subseteq \setN \times \setN.
\]
Each set $D_p$ is isomorphic to $\setN \times \setN$ in a very effective way. 
 For each $k$, when
the above construction says to place a point of $S_k$ at location $(i,j)$, we  instead place the point
at location $(p^{i+1}, p^{j+1})$ in $D_{p}$ where $p$ is the $(k+1)$st prime.
We do this simultaneously for all $k$, producing a set $S$ that is a subset of 
$\bigcup_{p \text{ prime}} D_p$. In this way, no point in a set $S_k$ is ever in 
the same row or column as a point of $S_{l}$ when $k \not = l$. Hence each row or column
of $S$ contains at most two points. Moreover, for each $i$ we can effectively enumerate
the finite set of points that could be included in row $i$ or column $i$. This means that
the bipartite graph corresponding to $S$ is bounded.

By assumption,
we can partition $S$ into two partial transversals.  We let $A$ be the set of $k$ so that
$e_1^k$ and $e_3^k$ are in the same block of the partition. Then $A$ will be the desired separating set. By construction, if $k$ is in the range of $f$ then $e_1^k$ and $e_3^k$
must be in the same block, while if $k$ is in the range of $g$ then $e_1^k$ and $e_3^k$ 
must be in different blocks.
\end{proof}

The previous proof can be adapted to the problem of decomposing $n$-regular bipartite graphs into complete matchings.  This Reverse Mathematics
result is originally due to Hirst~\cite{Hirst90}; also see Shafer~\cite[Theorem~6.1.8]{Shafer11}.

\begin{theorem}\label{thm:match}
The following are equivalent over $\RCAo$:
\begin{enumerate}
\item $\WKLo$.
\item If $G$ is a countable regular bipartite graph with degree $n$, there is a decomposition of the
edges of $G$ into $n$ complete matchings. 
\item If $G$ is a countable $2$-regular bipartite graph, there is a decomposition of the edges of $G$ into $2$ complete matchings. 
\end{enumerate}
\end{theorem}

\begin{proof}
In $\WKLo$, Theorem~\ref{thm:1} shows that we can find an edge coloring of $G$ with 
$n$ colors. In a regular graph with degree $n$, every vertex must have an edge of every color. 
Thus each color corresponds to a complete matching. This shows that (1) implies~(2).
Statement (2) directly implies statement~(3).

To show that (3) implies~$\WKLo$ over $\RCAo$, we only need to modify the proof of Theorem~\ref{thm:1} so that the graph is 2-regular. This means that we need to ensure that each row and each column in the construction of $S$ has exactly two points.   In that construction, the only rows or columns that have one point are the endpoints of the staircases. We can modify the construction so that
a staircase that is not capped will extend forever in each direction.  This means that we extend the staircase for $e^k_1$ down and to the left, and we extend each staircase up and to the right unless it is capped.  The resulting set $S$ has exactly two points in each row and each column, so the corresponding bipartite graph is 2-regular. The remainder of the reversal is the same as~Theorem~\ref{thm:1}.
\end{proof}

% The next corollary reduces this upper bound. 
%
%\begin{proposition} Every computable bipartite graph with maximal degree $n = 2m$ has a computable
%edge coloring with $3m$ colors. 
%\end{proposition}
%\begin{proof}
%
%
%\end{proof}

%
%Because $2m + 1 < 2(m+1)$, we can state a trivial corollary when the degree bound is odd. 
%
%\begin{corollary} Every computable bipartite graph with maximal degree $2m+1$ has a computable
%edge coloring with $3m+3$ colors. 
%\end{corollary}
%

A theorem of Kierstead shows that computable edge colorings of highly computable bipartite graphs can be 
obtained with one additional color.  

\begin{theorem}[Kierstead~\cite{Kierstead81}]
Every highly computable graph that has an edge coloring with $k$ colors has a computable edge coloring with $k+1$ colors.
\end{theorem}

\begin{corollary}
Every highly computable bipartite graph with maximal degree $k$ has a computable edge coloring with $k+1$ colors.
\end{corollary}
 
Hence the number of colors permitted in the coloring of a highly computable bipartite graph is a key factor.  When the maximal degree is $k$, a computable coloring with $k+1$ colors always exists,
but not necessarily a computable edge coloring with $k$ colors. 

\section{Questions}

Vizing's theorem shows that a countable graph in which the maximum degree is bounded by 
$n$ must have an edge coloring with at most $n+1$ colors.  It then follows from Kierstead's theorem 
that there must be a computable edge coloring with $n+2$ colors. 

The reversals above rely crucially on the fact that the degree of each vertex is the same
as the number of colors, so every color is used at every vertex. In Vizing's theorem one 
more color is allowed, so no vertex can use every color.  

We verified that Vizing's theorem follows from $\WKLo$ but did not produce a lower bound on the strength. Schmerl~\cite{Schmerl} originally posed this question in 1985 in the context of computable graph theory.

\begin{question}[Schmerl~\cite{Schmerl}]
 Is Vizing's theorem computably true?  In particular, is every computable $n$-regular graph edge colorable with $n+1$ colors? 
 \end{question}

In the same paper, Schmerl gives an example of a computable $3$-regular graph which is not computably edge colorable with $3$~colors.

When we consider regular bipartitie graphs, K\"onig's edge coloring theorem theorem and Hall's matching theorem coincide. Shafer~\cite{Shafer11} studied a related theorem of Birkhoff that a finite $n \times n$ doubly stochastic matrix is a convex linear combination of $n \times n$ permutation matrices.  Isbell proved a generalization of Birkhoff's theorem to the countable case.  Shafer showed that this countable version of Birkhoff's theorem is provable in $\WKLo$, and posed the following question.

\begin{question}[Shafer~\cite{Shafer11}]
What is the strength of the countable version of Birkhoff's theorem?
\end{question}

\bibliographystyle{amsplain}
\small

\providecommand{\bysame}{\leavevmode\hbox to3em{\hrulefill}\thinspace}
\providecommand{\MR}{\relax\ifhmode\unskip\space\fi MR }
\providecommand{\MRhref}[2]{%
  \href{http://www.ams.org/mathscinet-getitem?mr=#1}{#2}
}
\providecommand{\href}[2]{#2}

\end{document}